\documentclass[a4paper, 10pt, conference]{ieeeconf}      

\usepackage{graphics} 
\usepackage{epsfig} 
\usepackage{mathptmx} 
\usepackage{times} 
\usepackage{amsmath} 
\usepackage{amssymb}  
\usepackage{tikz}
\usepackage{float}
\usetikzlibrary{positioning,arrows,shapes,calc}
\usepackage{color}
\usepackage{caption}
\usepackage{subcaption}

\usepackage{amsthm}
\usepackage{txfonts}
\usepackage{array}

\theoremstyle{plain}

\newtheorem{theorem}{Theorem}

\theoremstyle{definition}
\newtheorem{pr}{Problem}
\newtheorem*{definition}{Definition}
\theoremstyle{remark}

\newcommand\R{\mathbb{R}}

\newcommand{\prenorm}[2]{\left\Vert #2 \right\Vert_{#1}}

\IEEEoverridecommandlockouts                              
\newcount\Comments  
\usepackage{color}
\definecolor{darkgreen}{rgb}{0,0.5,0}
\definecolor{purple}{rgb}{1,0,1}
\definecolor{darkred}{rgb}{0.5,0,0}
\newcommand{\kibitz}[2]{\ifnum\Comments=0\textcolor{#1}{#2}\fi}

\title{\LARGE \bf
Topology-based control design for congested areas in urban networks
}

\author{Liudmila Tumash, Carlos Canudas-de-Wit and Maria Laura Delle Monache
\thanks{This work was supported by the funding from the European Research Council (ERC) under the European Union’s
Horizon 2020 research and innovation programme (grant agreement 694209)}
\thanks{L. Tumash  is with Univ. Grenoble Alpes, CNRS, Inria, Grenoble INP, GIPSA-lab, 38000 Grenoble, France,
        {\tt\small liudmila.tumash@gipsa-lab.fr}}%
\thanks{C. Canudas-de-Wit is with Univ. Grenoble Alpes, CNRS, Inria, Grenoble INP, GIPSA-lab, 38000 Grenoble, France,
        {\tt\small carlos.canudas-de-wit@gipsa-lab.grenoble-inp.fr}}%
\thanks{M.~L. Delle Monache is with Univ. Grenoble Alpes, Inria, CNRS, Grenoble INP, GIPSA-Lab, 38000 Grenoble, France
        {\tt\small ml.dellemonache@inria.fr}}%
}

\begin{document}

\maketitle
\thispagestyle{empty}
\pagestyle{empty}

\begin{abstract}
This paper addresses the problem of a boundary control design for traffic evolving in a large urban network. The traffic state is described on a macroscopic scale and corresponds to the vehicle density, whose dynamics are governed by a two dimensional conservation law. We aim at designing a boundary control law such that the throughput of vehicles in a congested area is maximized. Thereby, the only knowledge we use is the network's topology, capacities of its roads and speed limits. In order to achieve this goal, we treat a 2D equation as a set of 1D equations by introducing curvilinear coordinates satisfying special properties. The theoretical results are verified on a numerical example, where an initially fully congested area is driven to a state with maximum possible throughput. 

\end{abstract}

\section{INTRODUCTION}

Rapidly growing urban areas cause heavy traffic congestions that negatively impact traffic mobility and environment, which makes traffic management an important issue to study. The first attempt to understand traffic was made in the fifties, as the kinematic wave theory has been introduced by Ligthill and Whitham \cite{LIG55} and, independently, Richards \cite{RIC56} (LWR model). This fluidodynamic model prescribes the conservation of the number of vehicles and describes the spatio-temporal evolution of vehicle density on a highway road. Further, this model was extended to the network level in \cite{HOL95} that represented a network as a set of edges (roads) and nodes (junctions). The classical traffic control strategies aim at improving the overall network efficiency and include ramp metering \cite{PAP91,REI15}, variable speed limits \cite{PAP08,MON17}, and control through route guidance\cite{HAL93} (see \cite{PAP03} for a review).

In case of large-scale urban networks, traffic modelling becomes a difficult task requiring macroscopic approaches due to increasing computational complexity. The first demonstration of a macroscopic relationship between density and flow should be recognized to \cite{WIL87}, who used data from microsimulations. Later this relation was also observed in the congested region of Yokohama, Japan \cite{GER08,DAG08}. The discovery of MFD (macroscopic fundamental diagram) gave rise to reservoir models, which track the number of cars in an urban area. However, if there is a large variance of densities in this area, then it should be partitioned into multiple zones in order to have a well-defined MFD \cite{HAJ13,LEC15}. Several control tasks have been posed and solved for MFD systems, e.g., see \cite{GER13,ABO13} that consider optimal perimeter control problems between different regions of urban network. 

However, control design in system described by MFD requires collection of large amount of data, as well as it can not correctly describe the variation of different traffic flow regimes as in was discussed in \cite{TUM19c}. Another way to model traffic in large-scale networks is to use a continuous two dimensional model known as 2D-LWR \cite{MOL18a}. This model describes the spatio-temporal evolution of vehicle density on a 2D plane. It includes the space-dependence of a fundamental diagram in order to capture the network infrastructure.
\vskip 4pt

\begin{minipage}[h]{0.2\textwidth}

\begin{tikzpicture}


\draw[black,ultra thick,-] (0,3) -- (3,3) -- (3,0) -- (0,0) -- (0,3);
\draw[-] (1,2.3) -- (2,2.3) -- (2,0.7) -- (1,0.7) -- (1,2.3);

\draw[-,fill,gray] (1.25,2.3) -- (1.35,2.3) -- (1.35,0.7) -- (1.25,0.7) -- (1.25,2.3);
\draw[-,fill,gray] (1.65,2.3) -- (1.75,2.3) -- (1.75,0.7) -- (1.65,0.7) -- (1.65,2.3);

\draw[-,fill,gray] (1,1.0) -- (1,1.1) -- (2,1.1) -- (2,1.0) -- (1,1.0);
\draw[-,fill,gray] (1,1.45) -- (1,1.55) -- (2,1.55) -- (2,1.45) -- (1,1.45);
\draw[-,fill,gray] (1,1.9) -- (1,2.0) -- (2,2.0) -- (2,1.9) -- (1,1.9);

\draw[ultra thick,->] (0.75,0.45) -- (0.95,0.65);
\draw[ultra thick,->] (1.08,0.45) -- (1.28,0.65);
\draw[ultra thick,->] (1.41,0.45) -- (1.61,0.65);

\draw[ultra thick,->] (0.75,0.78) -- (0.95,0.98);
\draw[ultra thick,->] (0.75,1.11) -- (0.95,1.31);
\draw[ultra thick,->] (0.75,1.44) -- (0.95,1.64);
\draw[ultra thick,->] (0.75,1.77) -- (0.95,1.97);

\node[below, scale=0.8] at (0.75,0.45) {demand};
\draw[ultra thick,-,darkred] (1,2.3) -- (2,2.3) -- (2,0.7);
\node[above, scale=0.8,darkred] at (2,2.3) {control};

\end{tikzpicture}
\end{minipage}
\hfill
\begin{minipage}[r]{0.26\textwidth}
In this paper, we consider a large-scale urban network with unidirectional roads, where the traffic dynamics are governed by the 2D-LWR model. This network will include congested areas, which we want to control from the boundary such that the
\end{minipage}
\vskip 4pt
\noindent maximal throughput of the traffic flow is achieved in the steady state. The stabilized system will be characterised by a reduced average latency and an increased average velocity. Our main contribution is to suggest a technique for the control design, for which we only need the knowledge of the network topology and its infrastructure, i.e. the maximal speeds and roads' capacities. This is the first work of this kind for two-dimensional traffic systems providing an explicit solution to the problem.

\section{Motivation}

\subsection{2D-LWR model}

The evolution of traffic in a large urban network can be described with the 2D-LWR model (\cite{MOL18a}), which is a two-dimensional conservation law, and the state corresponds to the vehicle density.

Let $D$ be a compact domain of the system corresponding to the considered urban area and let $\Gamma$ its boundary. We fix the initial condition $\rho_0(x,y)$ and the boundary flows $\phi_{in}(x,y,t)$ and $\phi_{out}(x,y,t)$ such that $\rho(x,y,t)$ : $D \times \mathbb{R}^+ \to \mathbb{R}^+$. Then, the initial boundary value problem (IBVP) for a system with dynamics governed by a 2D-LWR system reads:
\begin{equation}\label{IBVP}
\left\{
\begin{aligned}
& \frac{\partial \rho (x,y,t)}{\partial t} + \nabla \cdot \vec{\Phi} (x,y,\rho(x,y,t)) = 0, \\
& \vec{\Phi}(x,y,t) = \phi_{in}(x,y,t) \vec{d}_{\theta}(x,y), \quad \forall (x,y) \in \Gamma_{in} \\
& \vec{\Phi}(x,y,t) = \phi_{out}(x,y,t) \vec{d}_{\theta}(x,y), \quad \forall (x,y) \in \Gamma_{out} \\
& \rho (x,y,0) = \rho_0(x,y), 
\end{aligned} \right.
\end{equation}
where 
\begin{equation}\label{Flux_vector}
\vec{\Phi} = \Phi(x,y,\rho) \vec{d}_{\theta} (x,y)
\end{equation}
is a space-dependent flux function and 
\begin{equation}\label{Flux_direction}
\begin{aligned}
\vec{d}_\theta = &
\begin{pmatrix}
\cos(\theta(x,y)) \\
\sin(\theta(x,y))
\end{pmatrix}
\end{aligned}
\end{equation}
is the direction field set by the network's geometry. Further, $ \Gamma_{in} \subset \Gamma$ is a set of boundary points $(x,y)$ for which $\vec{n}(x,y) \cdot \vec{d}_{\theta} (x,y) > 0$, where $\vec{n}(x,y)$ is a unit normal vector to $\Gamma_{in}$ oriented inside $D$. Similarly, $\Gamma_{out} \subset \Gamma$ such that $\forall (x,y)\in \Gamma_{out}$ : $\vec{n}(x,y) \cdot \vec{d}_{\theta} (x,y) < 0$.

In \eqref{Flux_vector} the flow magnitude $\Phi (x,y,\rho) : [0, \rho_{max}(x,y)] \rightarrow \mathbb{R}^+$ is a strictly concave function with a unique maximum $\phi_{max}(x)$ $\forall (x,y) \in D$ (road \textit{capacity}) achieved at the \textit{critical density} $\rho_c (x,y)$, while the minimum is achieved twice, i.e., $\Phi(x,y,0) = \Phi(x,y,\rho_{max}) = 0$. The value of $\Phi (x,y,\rho)$ is determined by the \textit{fundamental diagram}, which relates flow and density of the system. We distinguish two different density regimes $ \forall (x,y) \in D$: $\Omega_f := [0, \rho_c(x,y)]$ indicates the \textit{free-flow regime} (vehicles move freely with positive kinematic wave speed), and $\Omega_c := (\rho_c(x,y), \rho_{max}(x,y)]$ is the \textit{congested regime} (negative kinematic wave speed), see Fig. \ref{fig_FD}. In this paper, we will consider the so-called \textit{Greenshields FD}, see \cite{GRE35}:

\begin{equation}\label{Flux}
\Phi \left(x, y, \rho (x,y,t) \right) = \varv_{max}(x,y) \left(1 - \frac{\rho (x,y,t)}{\rho_{max}(x,y)} \right) \rho (x,y,t). 
\end{equation}

Note that in \eqref{Flux} the function achieves $\phi_{max}$ at $\rho_c = \rho_{max}/2$.

\begin{figure}
\begin{center}
\begin{tikzpicture}

\node [below] at (5,0) {$\rho$};
\node [left] at (0,3.7) {$\Phi(\rho)$};

\path[thick,-]
(0,0) edge[out=85,in=185,looseness=1.0,fill=green!20] (2,3);
\path[fill=green!20] (0,0) -- (2,3) -- (2,0);

\path[thick,-]
(2,3) edge [out=355,in=95,looseness=1.0,fill=red!20] (4,0);
\path[fill=red!20] (2,3) -- (4,0) -- (2,0);
\path[fill=red!20] (2,3) -- (4,0);

\draw [<->] (0,3.7) -- (0,0) -- (5,0);

\draw[fill] (2,3) circle [radius=0.06];

\node at (1.2,1) {$\Omega_f$};
\node at (3,1) {$\Omega_c$};
\draw [dashed] (2,3) -- (2,0);
\draw [dashed] (2,3) -- (0,3);
\node [below] at (2,0) {$\rho_c = \frac{\rho_{max}}{2}$};
\node [left] at (0,3) {$\phi_{max}$};
\node [below left] at (0,0) {$0$};
\node [below] at (4,0) {$\rho_{max}$};
\end{tikzpicture}
\end{center}
\caption{Greenshields FD with free-flow regime denoted by $\Omega_f$ (in green) and congested regime denoted by $\Omega_c$ (in red).} \label{fig_FD}
\end{figure}
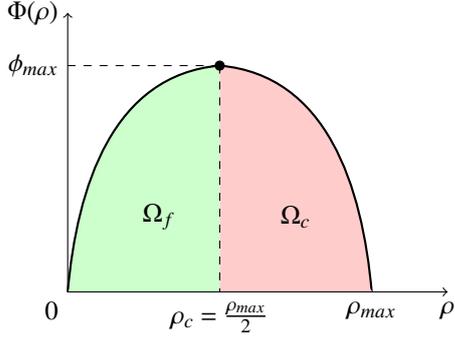

The in- and outflows $\phi_{in} (t)$ and $\phi_{out} (t)$ in \eqref{IBVP} are 
\begin{equation} \label{phi_numflux}
\left\{
\begin{aligned}
& \phi_{in} (x,y,t) = \min\left\{D_{in} (x,y,t), S \left(\rho (x,y,t) \right) \right\}, \quad (x,y) \in \Gamma_{in}  \\
& \phi_{out} (x,y,t) = \min\left\{D \left(\rho (x,y,t) \right), S_{out} (x,y,t) \right\}, \; \, (x,y) \in \Gamma_{out}.
\end{aligned} \right.
\end{equation}
where $D(\rho)$ and $S(\rho)$ are the demand and the supply functions defined as
\begin{equation}\label{EQ:demand}
D(\rho) = 
\left\{
\begin{aligned}
& \phi(\rho), \quad \text{if} \quad 0 \leq \rho \leq \rho_c, \\
& \phi_{max}, \quad \text{if} \quad \rho_c < \rho \leq \rho_{max},
\end{aligned} \right.
\end{equation}
\begin{equation}\label{EQ:supply}
S(\rho) = 
\left\{
\begin{aligned}
& \phi_{max}, \quad \text{if} \quad 0 \leq \rho \leq \rho_c, \\
& \phi(\rho), \quad \text{if} \quad \rho_c < \rho \leq \rho_{max}.
\end{aligned} \right.
\end{equation}

Finally, $D_{in} (t)$ and $S_{out} (t)$ are the demand function at the entry and the supply function at the exit of the road.

To reproduce the traffic's evolution on an urban network using \eqref{IBVP}, we need to construct $\forall (x,y) \in D$ the direction field $\vec{d}_{\theta}(x,y)$, the maximal density $\rho_{max}(x,y)$ and the maximal velocity $v_{max}(x,y)$. We do it by the Inverse Distance Weighting as described in \cite{TUM19c,MOL18a} assuming that the flux field depends on speed limits, thereby tuning a parameter measuring the sensitivity of the flux to the mutual location of roads in a network. The maximal density is reconstructed assuming that the network is filled completely (every $6m$) with vehicles, and that each vehicle contributes to the global density with a Gaussian kernel with standard deviation $d_0 = 100m$ centred at its position (see Section 4.1 of \cite{MOL18a}).

The flux field $\vec\Phi$ \eqref{Flux_vector} defined in \eqref{Flux_vector} depends on the state in the magnitude, but not in the direction of the flow. Therefore one can define stationary integral curves of the flux field that describe the paths along which the flow propagates.

\subsection{2D problem as a set of 1D problems}

Assume that we can perform coordinate transformation that translates integral curves of the flux field into a set of straight parallel lines, where each such line can be treated as a 1D system. We introduce the new coordinates $(\xi,\eta)$ from:

\begin{equation}\label{Jacobian}
\begin{pmatrix}
d \xi \\
d \eta
\end{pmatrix}
= C_{\theta}(x,y) R_{\theta}(x,y)
\begin{pmatrix}
dx \\
dy
\end{pmatrix}
\end{equation}
where $R_{\theta}(x,y)$ is the rotation matrix used to rotate the integral lines in $(x,y)$-plane and $C_{\theta}(x,y)$ is the scaling matrix used to make these lines having the same metric as illustrated in Fig. \ref{fig_trafo} (it will be explained in more details in Section III, B).

\begin{figure}
\begin{center}
\begin{tikzpicture}

\draw[thick,->] (0,0) -- (0,3.5);
\node[left] at (0,1.5) {$y$};
\draw[-,dashed] (0,3) -- (3,3) -- (3,0);
\draw[thick,->] (0,0) -- (3.3,0);
\node[below] at (1.5,0) {$x$};

\node[above left] at (0,3) {$(a)$};
\draw[-, dashed] (4.7,1.5) -- (6.1,0) -- (6.2,0) -- (7.6,1.5) -- (6.15,3.1) -- (4.7,1.5);
\draw[thick,->] (4.7,0) -- (4.7,3.5);
\node[above left] at (4.7,3) {$(b)$};
\node[left] at (4.7,2.7) {$\eta$};
\draw[thick,->] (4.7,0) -- (7.9,0);
\node[below] at (6.3,0) {$\xi$};

\draw[darkgreen,ultra thick,-] (5.7875,0.3) -- (6.5125,0.3);
\node at (5.7,0.2) {$1$};
\draw[darkgreen,ultra thick,-] (5.425,0.7) -- (6.875,0.7);
\node at (5.3,0.6) {$2$};
\draw[darkgreen,ultra thick,-] (5.0625,1.1) -- (7.2375,1.1);
\node at (4.9,1) {$3$};
\draw[darkgreen,ultra thick,-] (4.7,1.5) -- (7.6,1.5);
\node[left] at (4.7,1.5) {$4$};
\draw[darkgreen,ultra thick,-] (5.0625,1.9) -- (7.2375,1.9);
\node[above] at (4.9,1.8) {$5$};
\draw[darkgreen,ultra thick,-] (5.425,2.3) -- (6.875,2.3);
\node[above] at (5.3,2.2) {$6$};
\draw[darkgreen,ultra thick,-] (5.7875,2.7) -- (6.5125,2.7);
\node[above] at (5.7,2.6) {$7$};

\draw[darkgreen,ultra thick,-] (0,1.8) .. controls (1,2) and (1.4,2.3) .. (1.9,3);
\node[above] at (0.2,1.8) {$6$};

\draw[darkgreen,ultra thick,-] (0,2.4) .. controls (0.3,2.5) and (0.9,2.6) .. (1.3,3);
\node[above] at (0.2,2.4) {$7$};

\draw[darkgreen,ultra thick,-] (0,0) .. controls (1,2) and (2,1) .. (3,3);
\node[above] at (0.2,0.4) {$4$};
\draw[darkgreen,ultra thick,-] (0,1) .. controls (1.5,2.5) and (2,1.2) .. (2.5,3);
\node[above] at (0.2,1.18) {$5$};
\draw[darkgreen,ultra thick,-] (0.8,0) .. controls (1.5,1.5) and (2,1.3) .. (3,2.5);
\node[above] at (0.7,0.0) {$3$};
\draw[darkgreen,ultra thick,-] (1.5,0) .. controls (1.6,1.2) and (2.5,1.4) .. (3,2);
\node[above] at (1.35,0) {$2$};
\draw[darkgreen,ultra thick,-] (2.2,0) .. controls (2.4,0.9) and (2.7,1.3) .. (3,1.5);
\node[above] at (2.1,0) {$1$};

\path[thick,->]
(3,1.8) edge [out=60,in=120,looseness=1.0] (4.6,1.8);
\node[below] at (3.8,2.1) {$C_{\theta}R_{\theta}$};

\end{tikzpicture}
\end{center}
\caption{Coordinate transformation mapping curved trajectories (a) into straight lines (b) having the same metric.} \label{fig_trafo}
\end{figure}
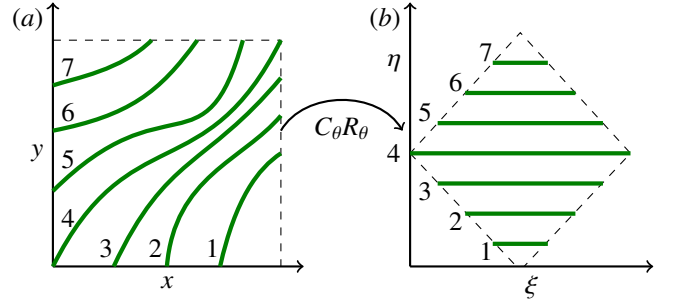

In case of straight integral curves (see Fig. \ref{fig_trafo}b) we do not need to perform neither rotation nor scaling, i.e. $\theta = 0$ $\forall (x,y) \in D$ and $C_\theta = R_{\theta} = \mathbb{I}$. Let us first pose and solve the control problem for this case. Then, we will show that this problem is similarly solved for the case of curved lines due to the coordinate transformation driving (a) to (b) in Fig. \ref{fig_trafo}.

\section{Control design}

In this section we first determine what is the desired state $\rho_d (x,y)$ providing the maximal throughput of the system, and then we pose and solve the control problem.


\subsection{Case of straight lines}

We rewrite \eqref{IBVP} in coordinates $(\xi,\eta)$, which by \eqref{Jacobian} coincide with $(x,y)$ if $\forall (x,y) \in D$ $\theta(x,y) = 0$ (Fig. \ref{fig_trafo}b). The vector field becomes $\vec{d}_{\theta} (\xi,\eta) = (1, 0)$, and by \eqref{Flux_vector}:
\begin{equation}\label{Phi_0}
\vec{\Phi} = \Phi(\xi,\eta,\rho) \begin{pmatrix}
1 \\
0
\end{pmatrix},
\end{equation}
and using \eqref{Phi_0} we obtain for the nabla operator in \eqref{IBVP}: 
\begin{equation}\label{nabla_0}
\begin{pmatrix}
\frac{\partial}{\partial \xi}, &\frac{\partial}{\partial \eta}
\end{pmatrix}
\begin{pmatrix}
1 \\ 
0
\end{pmatrix}
\Phi(\xi,\eta,\rho) = \frac{\partial \Phi(\xi,\eta,\rho)}{\partial \xi}.
\end{equation}

We are now ready to rewrite \eqref{IBVP} using \eqref{Phi_0} and \eqref{nabla_0} as

\begin{equation}\label{IBVP_0}
\left\{
\begin{aligned}
& \frac{\partial \rho (\xi,\eta,t)}{\partial t} +  \frac{\partial \Phi \left(\xi,\eta,\rho \right)}{\partial \xi} = 0, \\
& \phi_{in}(\eta,t) = \min\left\{D_{in} (\eta,t), S \left(\rho \left(\xi_{min}(\eta),\eta,t \right) \right) \right\}, \\
& \phi_{out}(\eta,t) = \min\left\{D \left(\rho \left(\xi_{max}(\eta),\eta,t \right) \right), S_{out} \left(\eta, t \right) \right\},  \\
& \rho (\xi,\eta,0) = \rho_0(\xi,\eta), 
\end{aligned} \right.
\end{equation} 
where 
$$
\xi_{min}(\eta) = \min\limits_{\substack{(x,y)\in D, \\ \eta(x,y)=\eta }} \xi(x,y),
\quad \xi_{max}(\eta) = \max\limits_{\substack{(x,y)\in D, \\ \eta(x,y)=\eta }} \xi(x,y). 
$$

The system \eqref{IBVP_0} is a continuous set of 1D-LWR equations each following the path numbered by $\eta$.

\vskip 5pt

\subsubsection{Steady-state}

Our goal is to drive the system to the steady-state providing the maximal throughput of the system. By \eqref{IBVP_0}, the steady-state $\rho^*(\xi,\eta)$ implies space-independence of $\phi^*(\eta)$, and can be achieved only for stationary $D_{in}(\eta)$ and $S_{out}(\eta)$. The steady-state flow can be obtained in analogy with \cite{WU14} who analysed it for a single ring road with bottlenecks, i.e., $\forall \eta \in D$ the following holds
\begin{equation}\label{ss_flow}
\phi^*(\eta) = \min \left\{D_{in} (\eta), \min\limits_{\xi \in [\xi_{min}(\eta),\xi_{max}(\eta)]} \phi_{max} (\xi,\eta), S_{out} (\eta) \right\} \quad 
\end{equation}  

Thus, the steady-state flow is the minimum between demand at the entry, supply at the exit and the minimal capacity along the path $\eta$, which is related to the strongest bottleneck. By bottlenecks we mean permanent capacity constraints in the network itself, i.e., the road segment with low speed limit or with fewer lanes (see Fig. \ref{fig_bottleneck}). 

We also need to extract the correct density $\rho^*$ from \eqref{ss_flow}, since there are two density values (one in free-flow regime, one in congested regime) that correspond to the same flow. Let us fix some $\eta$ and denote the location of the strongest bottleneck by $\xi^*$. We assume that the inflow $D_{in}$ is always larger than the capacity at the bottleneck. Then, if $S_{out}$ is also larger, according to \cite{WU14}, the steady-state $\rho^*$ corresponds to the congested regime $\forall \xi \in [\xi_{min},\xi^*)$, and then the free-flow regime occurs $\forall \xi \in (\xi^*,\xi_{max}]$. If there are several such $\xi^*$ (or it is an interval), then we take the left-most value, i.e., $\xi^* = \xi^*_1$ in Fig. \ref{fig_bottleneck}. If $S_{out}$ is smaller than the capacity at the bottleneck, the whole domain will be congested.

\subsubsection{Optimal equilibrium manifolds}

Here we mainly consider congested urban areas, which appear if the demand at its upstream boundary is too high, i.e. $D_{in} \geq \phi_{max}$. Thus, the minimum function in \eqref{phi_numflux} is always resolved to the supply function at the domain exit, which becomes the control variable, i.e., $u(\eta) = S_{out} (\eta)$ $\forall \eta \in D$. By \eqref{ss_flow} this means that in order to provide the maximal throughput of the system we need to set $S_{out} (\eta) = \phi_{max} \left(\xi^*(\eta),\eta \right)$. However, this control will not be accepted by the system, since the traffic will be in the free-flow regime $\forall \xi \in [\xi^*(\eta),\xi_{max}(\eta)]$ (see \eqref{ss_flow} and the discussion above).
Thus, we introduce some small constant $\epsilon > 0$ that needs to be subtracted from the maximal possible flow along the path in the equilibrium $\forall \eta \in D$:
\begin{equation}\label{desired_flow}
\phi_d \left(\eta \right) = \phi_{max} \left(\xi^*(\eta), \eta \right) - \epsilon.
\end{equation}

By setting $S_{out}(\eta) = \phi_d (\eta)$ we translate the bottleneck location from $\xi^*(\eta)$ to $\xi_{max}(\eta)$, and then the congested regime will capture the whole interval $[\xi_{min}(\eta),\xi_{max}(\eta)]$ (see Fig. \ref{fig_bottleneck}). This allows us to control the system from the exit. From the practical viewpoint, subtraction of $\epsilon$ does not change much the desired state, since $\epsilon$ can be set to an arbitrarily small value. Thus, in the following we will call the desired state an $\epsilon$-optimal state w.r.t. throughput maximization. Note that controlling the domain exit can be physically realized by installing, e.g., traffic lights.

\begin{figure}
\begin{center}
\begin{tikzpicture}

\node [below] at (7.5,0) {$L$};
\draw [->] (0,0) -- (7.5,0);
\path[fill=black!30,thick,-] (0,0.5) -- (2.6,0.5) -- (2.6,1) -- (4.6,1) -- (4.6,0.5) -- (7.5,0.5) -- (7.5, 1.5) -- (0,1.5) -- (0,0.5);
\draw [thick,-] (0,0.5) -- (2.6,0.5);
\draw [thick,-] (2.6,0.5) -- (2.6,1);
\draw [thick,-] (2.6,1) -- (4.6,1);
\draw [thick,-] (4.6,1) -- (4.6,0.5);
\draw [thick,-] (4.6,0.5) -- (7.5,0.5);
\draw [thick,-] (0,1.5) -- (7.5,1.5);
\draw [thick,-] (0,1.5) -- (0,0.5);
\draw [thick,-] (7.5,1.5) -- (7.5,0.5);
\draw [dashed] (2.6,0.5) -- (2.6,-0.2);
\node [below] at (2.6,0) {$\xi^*_1$};
\draw [dashed] (4.6,0.5) -- (4.6,-0.2);
\node [below] at (4.6,0) {$\xi^*_2$};
\node [below left] at (0,0) {$0$};

\path[thick,-]
(0.3,2.3) edge [out=80,in=180,looseness=1.0,fill=green!20] (1.3,4.1);
\path[fill=green!20] (0.3,2.3) -- (1.3,4.1) -- (1.3,2.3);
\path[thick,-]
(1.3,4.1) edge [out=355,in=100,looseness=1.0,fill=red!20] (2.3,2.3);
\path[fill=red!20] (1.3,2.3) -- (1.3,4.1) -- (2.3,2.3);
\path[fill=red!20] (1.3,4.1) -- (2.3,2.3);
\draw[fill] (1.85,3.83) circle [radius=0.06];
\draw[thick,->] (0.3,2.3) -- (2.5,2.3);
\node [below] at (1.4,2.3) {$\rho$};
\draw[thick,->] (0.3,2.3) -- (0.3,4.7);
\node [left] at (0.3,4.5) {$\Phi(\rho)$};

\path[thick,-]
(2.7,2.3) edge [out=80,in=180,looseness=1.0,fill=green!20] (3.5,3.9);
\path[fill=green!20] (2.7,2.3) -- (3.5,3.9) -- (3.5,2.3);
\path[thick,-]
(3.5,3.9) edge [out=355,in=100,looseness=1.0,fill=red!20] (4.3,2.3);
\path[fill=red!20] (3.5,3.9) -- (3.5,2.3) -- (4.3,2.3);
\path[fill=red!20] (3.5,3.9) -- (4.3,2.3);
\draw[fill] (3.71,3.83) circle [radius=0.06];
\node[above] at (3.71,3.83) {$\phi_{max} (\xi^*) - \epsilon$};
\draw[thick,->] (2.7,2.3) -- (4.5,2.3);
\node [below] at (3.6,2.3) {$\rho$};
\draw[thick,->] (2.7,2.3) -- (2.7,4.7);
\node [left] at (2.7,4.5) {$\Phi(\rho)$};

\path[thick,-] 
(4.7,2.3) edge [out=80,in=185,looseness=1.0,fill=green!20] (5.8,4.2);
\path[fill=green!20] (4.7,2.3) -- (5.8,4.2) -- (5.8,2.3);

\path[thick,-]
(5.8,4.2) edge [out=355,in=100,looseness=1.0,fill=red!20] (6.9,2.3);

\path[fill=red!20] (5.8,4.2) -- (6.9,2.3) -- (5.8,2.3);
\path[fill=red!20] (5.8,4.2) -- (6.9,2.3);
\draw[fill] (6.45,3.83) circle [radius=0.06];
\draw[thick,->] (4.7,2.3) -- (7.1,2.3);
\node [below] at (5.9,2.3) {$\rho$};
\draw[thick,->] (4.7,2.3) -- (4.7,4.7);
\node [left] at (4.7,4.5) {$\Phi(\rho)$};

\draw [dashed] (0,3.83) -- (7.5,3.83);

\node [above] at (7.5,3.83) {$S_{out}$};

\end{tikzpicture}
\end{center}
\caption{Single road with corresponding fundamental diagrams and with a bottleneck located in the sketch $\xi^* = [\xi^*_1, \xi^*_2]$.} \label{fig_bottleneck}
\end{figure}
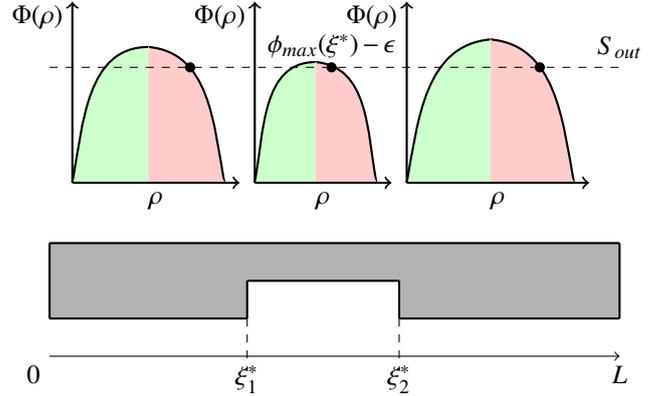

\begin{definition}
The desired $\epsilon$-optimal equilibrium state $\rho_d (\xi,\eta)$ w.r.t. the throughput maximization is defined $\forall \eta \in D$ as
\begin{equation}\label{desired_density}
\rho_d (\xi,\eta) = \frac{\rho_{max} (\xi,\eta)}{2} + \sqrt{\frac{\rho_{max}^2 (\xi,\eta)}{4} - \frac{\rho_{max} (\xi,\eta)}{\varv_{max} (\xi,\eta)} \phi_d(\eta)}, 
\end{equation}
where $\phi_d(\eta)$ is defined in \eqref{desired_flow} and $\epsilon > 0$, see Fig. \ref{fig_bottleneck}.
\end{definition}

Note that \eqref{desired_density} was obtained by taking the inverse of \eqref{Flux} for $\Phi(\rho_d) = \phi_d$ with the plus sign for the congested regime. 

\vskip 5pt

\subsubsection{Control design for straight lines}

Let us define the density error 
\begin{equation}\label{density_error}
\tilde{\rho}\left(\xi, \eta, t \right) = \rho \left(\xi, \eta, t \right) - \rho_d \left( \xi, \eta \right),
\end{equation}
 
and introduce the $L_2$-norm for the density error as
\begin{equation*}
\prenorm{2}{\tilde{\rho}(\xi,\eta,t)} := \sqrt{ \int\limits_{\eta_{min}}^{\eta_{max}} \int\limits_{\xi_{min} (\eta)}^{\xi_{max} (\eta)}  \tilde{\rho}^2(\xi,\eta,t) d \xi d \eta}, \quad \forall t \in \R^+,
\end{equation*}
where
$$
\eta_{min} = \min\limits_{(x,y)\in D} \eta(x,y), \quad \eta_{max} = \max\limits_{(x,y)\in D} \eta(x,y).
$$

\begin{pr}\label{Problem1}

Given network functions $\varv_{max}(\xi,\eta)$ and $\rho_{max}(\xi,\eta)$ $\forall (\xi,\eta) \in D$, initially congested traffic $\rho_0(\xi,\eta) \in (\rho_d (\xi,\eta), \rho_{max}(\xi,\eta)]$ with dynamics governed by \eqref{IBVP_0}, and given large constant demand at domain entry $D_{in}(\eta)$ $\forall \eta \in D$, find a boundary control $u(\eta) = S_{out} (\eta)$ such that $\forall (\xi, \eta) \in D$:

\begin{equation}\label{goal}
\lim \limits_{t \to \infty} \prenorm{2}{\tilde{\rho}(\xi,\eta,t)} = 0.   
\end{equation}
\end{pr}

\begin{theorem}
The goal defined in Problem 1 is achieved with
\begin{equation}\label{Result}
u(\eta) = S_{out}(\eta)  = \phi_{max} \left( \xi^*(\eta), \eta \right) - \epsilon \quad \forall \eta \in D,
\end{equation}
where $\phi_{max} \left( \xi^*(\eta), \eta \right) = \min\limits_{\xi \in [\xi_{min}(\eta),\xi_{max}(\eta)]} \phi_{max}\left(\xi, \eta \right)$.
\end{theorem}

\begin{proof}
We define the Lyapunov function candidate $\forall \eta \in D$
\begin{equation}\label{Lyapunov_Function}
V(t) = \frac{1}{2} \int\limits_{\xi_{min}}^{\xi_{max}} e^{\xi} \tilde{\rho}^2(\xi,\eta,t) d\xi,
\end{equation}
where $e^{\xi}$ is used as a weighting function. For simplicity of notations, we neglect variable $\eta$ as an argument. The time derivative of \eqref{Lyapunov_Function} is
\begin{equation}\label{LF_der1}
\dot{V}(t) = \int\limits_{\xi_{min}}^{\xi_{max}} e^{\xi} \tilde{\rho}(\xi,t)  \frac{\partial \tilde{\rho}(\xi,t)}{\partial t} d\xi.
\end{equation}

To simplify \eqref{LF_der1}, we use the time-independence of $ \rho_d$ as:
\begin{equation}\label{equality}
\frac{\partial \tilde{\rho}(\xi,t)}{\partial t} =  \frac{\partial \rho(\xi,t)}{\partial t} \equiv -\frac{\partial \Phi \left(\xi,\rho \right)}{\partial \xi} = -\frac{\partial \Phi \left(\xi, \rho_d + \tilde{\rho} \right)}{\partial \xi}.
\end{equation}

We consider the most right-hand-side term and linearise the flux function around the desired state as follows:

\begin{equation}\label{Linearisation}
\begin{aligned}
\Phi \left(\xi, \rho_d + \tilde{\rho} \right) \approx \Phi \left(\xi,\rho_d \right) + \frac{\partial \Phi \left( \xi, \rho_d \right)}{\partial \rho} \tilde{\rho} 
\end{aligned}
\end{equation}

Using time-independence of the first term on the right side of \eqref{Linearisation}, we obtain from \eqref{equality} that 
\begin{equation}\label{equality_2}
\frac{\partial \tilde{\rho}(\xi,t)}{\partial t} = -\frac{\partial \left( \Phi^{\prime} \left( \rho_d \right) \tilde{ \rho} \right)}{\partial \xi},
\end{equation}
where the prime denotes $\Phi^{\prime} = \partial \Phi / \partial \rho$.

Now we insert \eqref{equality_2} into \eqref{LF_der1} and get

\begin{equation}\label{LF_der2}
\begin{aligned}
& \dot{V}(t) = - \int\limits_{\xi_{min}}^{\xi_{max}} e^{\xi} \tilde{\rho}(\xi,t) \frac{\partial \left(\Phi^{\prime} \left(\rho_d \right) \tilde{\rho} \left(\xi,t \right) \right)}{\partial \xi} d\xi \\ & 
= - \int\limits_{\xi_{min}}^{\xi_{max}} e^{\xi} \frac{2 \Phi^{\prime} \left(\rho_d \right)}{2 \Phi^{\prime} \left(\rho_d \right)} \tilde{\rho}(\xi,t) \frac{\partial \left( \Phi^{\prime} \left(\rho_d \right) \tilde{\rho} \left(\xi,t \right) \right)}{\partial \xi} d\xi \\ & 
= - \int\limits_{\xi_{min}}^{\xi_{max}} \frac{e^\xi}{2 \Phi^{\prime} \left(\rho_d \right)} \frac{\partial \left(\Phi^{\prime} \left(\rho_d \right) \tilde{\rho} \left(\xi,t \right) \right)^2}{\partial \xi} d\xi.
\end{aligned}
\end{equation}

We now consider $\Phi^{\prime} \left( \rho_d \left( \xi \right) \right)$, which by \eqref{Flux} is defined as
\begin{equation}\label{Greenshields_der}
\Phi^{\prime} \left(\rho_d \left( \xi \right) \right) = v_{max} \left( \xi \right) \left( 1 - \frac{2 \rho_d \left( \xi \right)}{ \rho_{max} \left( \xi \right)} \right).
\end{equation}

Let us estimate the desired state at the bottleneck located at $\xi^*$,  using \eqref{Flux} and \eqref{desired_density}. By \eqref{Flux} with $ \rho_c = \rho_{max} / 2$, we get $ \phi_{max} = v_{max} \rho_{max} / 4$. Then, using \eqref{desired_density} we can write:

\begin{equation}\label{rho_des_bottleneck}
\begin{aligned}
& v_{max} \left( \xi^* \right) \rho_d \left( \xi^* \right) - \frac{v_{max} \left( \xi^* \right) \rho_{d}^{2} \left( \xi^* \right)}{\rho_{max} \left( \xi^* \right)} = \frac{v_{max} \left( \xi^* \right) \rho_{max} \left( \xi^* \right)}{4} - \epsilon \\ & \Rightarrow \rho_{d}^{2} \left( \xi^* \right) - \rho_d \left( \xi^* \right) \rho_{max} \left( \xi^* \right) + \frac{\rho_{max}^{2} \left( \xi^* \right)}{4} - \frac{\epsilon}{v_{max} \left( \xi^* \right)} = 0 \\ & \Rightarrow \rho_d \left( \xi^* \right) = \frac{\rho_{max} \left( \xi^* \right)}{2} + \sqrt{\frac{\epsilon}{ v_{max} \left( \xi^* \right)}} 
\end{aligned}
\end{equation}

In the latter expression we chose ''plus'' to provide the congested regime. Being a concave FD, its derivative $\Phi^{\prime}$ achieves its maximal value at the bottleneck (in the free-flow regime it is vice versa). Thus, we insert \eqref{rho_des_bottleneck} into \eqref{Greenshields_der} and introduce a variable $\nu$ used to denote $ \Phi^{\prime}$ at the bottleneck:
\begin{equation}\label{nu}
\Phi^{\prime} \left( \xi^* \right) = - \sqrt{v_{max} \left( \xi^* \right) \epsilon} = -\nu.
\end{equation}

To gain more insight, let us now again use $\eta$ as an argument. Thus, we can bound \eqref{LF_der2} from above using \eqref{nu}:

\begin{equation}\label{LF_der3}
\dot{V}(t) \leq \frac{1}{2 \nu} \int\limits_{\xi_{min}}^{\xi_{max}} e^{\xi} \frac{\partial \left(\Phi^{\prime} \left(\rho_d \left( \xi, \eta \right) \right) \tilde{\rho} \left(\xi, \eta, t \right) \right)^2}{\partial \xi} d\xi.
\end{equation}

Integration by parts of \eqref{LF_der3} yields
\begin{equation} \label{Int_by_parts}
\begin{aligned}
\dot{V}(t) = & \frac{e^{\xi_{max}}}{2 \nu} \Phi^{\prime 2} \left( \rho_d \left( \xi_{max}, \eta \right) \right) \tilde{\rho}^2 \left( \xi_{max}, \eta, t \right) \\ & - \frac{e^{\xi_{min}}}{2 \nu} \Phi^{\prime 2} \left( \rho_d \left( \xi_{min}, \eta \right) \right) \tilde{\rho}^2 \left( \xi_{min}, \eta, t \right) \\ & - \frac{1}{2 \nu} \int\limits_{\xi_{min}}^{\xi_{max}} e^{\xi} \left( \Phi^{\prime} \left( \rho_d  \left(\xi, \eta \right) \right) \tilde{\rho} \left( \xi, \eta, t \right) \right)^2 d\xi 
\end{aligned}
\end{equation}

The last term in \eqref{Int_by_parts} can be again bounded by $\nu$ as follows:

\begin{equation} \label{Integral_bound}
\begin{aligned}
 - \frac{1}{2 \nu} \int\limits_{\xi_{min}}^{\xi_{max}} e^{\xi} & \left( \Phi^{\prime} \left( \rho_d \left( \xi, \eta \right) \right) \tilde{\rho} \left( \xi, \eta, t \right) \right)^2 d\xi \\ & \leq  - \frac{\nu}{2} \int\limits_{\xi_{min}}^{\xi_{max}} e^{\xi} \tilde{\rho}^2 \left( \xi, \eta, t \right) d\xi  = -\nu V(t)
\end{aligned}
\end{equation}

Inserting \eqref{Integral_bound} into \eqref{Int_by_parts}, we see that the only positive term is the first one, which can be eliminated by setting $\tilde{\rho} \left( \xi_{max}, \eta, t \right) = 0$, i.e., $\rho\left(\xi_{max}, \eta, t \right) = \rho_d \left( \xi_{max}, \eta \right)$. In terms of control variables this is equivalent to
\begin{equation}\label{u_out}
u(\eta) = S_{out}(\eta)  = \phi_{max} \left( \xi^*(\eta), \eta \right) - \epsilon, \quad \forall \eta \in D.
\end{equation}
Note that the control term is different for each $\eta$. Thus, with \eqref{u_out} and \eqref{Integral_bound}, we can rewrite \eqref{Int_by_parts}

\begin{equation} \label{Int_by_parts_2}
\dot{V}(t) = - \frac{e^{\xi_{min}}}{2 \nu} \Phi^{\prime 2} \left( \rho_d \left( \xi_{min}, \eta \right) \right) \tilde{\rho}^2 \left( \xi_{min}, \eta, t \right) - \nu V(t).
\end{equation}

We have proved the $L_2$ convergence in $\xi$ coordinate of the state $\rho \left( \xi, \eta, t \right)$ to the desired state $\rho_d \left( \xi, \eta \right)$ as $t \to \infty$ $\forall \eta \in D$. Hence, it follows that the point-wise convergence in $\eta$ is also achieved, which implies the $L_{\infty}$ convergence in $\eta$. In bounded spaces (which is the case for $ \eta $-space) this also implies the $L_2$ convergence in $\eta$, which proves the asymptotic $L_2$ convergence in $\left(\xi, \eta \right)$-space. 
 
\end{proof}

\subsection{Coordinate transformation}

Here we explain the coordinate transformation that we mentioned in Section II, B. The system given by \eqref{IBVP} will be rewritten as a set of parametrised 1D equations \eqref{IBVP_0}. 

In \eqref{Jacobian} the rotation matrix is given by

\begin{equation}\label{Rotation_mat}
R_{\theta}(x,y) = 
\begin{pmatrix}
\cos \left( \theta(x,y) \right) &  \sin \left( \theta (x,y) \right) \\
- \sin \left( \theta (x,y) \right) & \cos \left( \theta(x,y) \right)
\end{pmatrix},
\end{equation}
and $C_{\theta}(x,y)$ is a diagonal scaling matrix given by
\begin{equation}\label{Scaling_matrix}
C_{\theta}(x,y) = 
\begin{pmatrix}
\alpha(x,y) & 0 \\
0 & \beta(x,y)
\end{pmatrix}
\end{equation}
where $\alpha(x,y)$ and $\beta(x,y)$ are positive and bounded scaling parameters used to normalize the metric in $(\xi,\eta)$-space and to make it uniformly distributed in space, and they satisfy:
\begin{equation}\label{alpha}
-\sin \theta \frac{\partial \left(\ln \alpha \right)}{\partial x} + \cos \theta \frac{\partial \left( \ln \alpha \right)}{\partial y} = \cos \theta \frac{\partial \theta}{\partial x} + \sin \theta \frac{\partial \theta}{\partial y}
\end{equation}
and 
\begin{equation}\label{beta}
\cos \theta \frac{\partial \left(\ln \beta\right)}{\partial x} + \sin \theta \frac{\partial \left( \ln \beta \right)}{\partial y} = \sin \theta \frac{\partial \theta}{\partial x} - \cos \theta \frac{\partial \theta}{\partial y}
\end{equation}

These PDEs \eqref{alpha} and \eqref{beta} come from the invarince of the order of taking partial derivatives of $\xi$ and $\eta$ w.r.t. $x$ and $y$ to provide independence of the chosen integration path (see \cite{TUM19c} for more details), i.e.

\begin{equation*}
\frac{\partial}{\partial y} \left( \frac{\partial \xi (x,y)}{\partial x} \right) = \frac{\partial}{\partial x} \left( \frac{\partial \xi(x,y)}{\partial y} \right),
\end{equation*}
and 
\begin{equation*}
\frac{\partial}{\partial y} \left( \frac{\partial \eta (x,y)}{\partial x} \right) = \frac{\partial}{\partial x} \left( \frac{\partial \eta(x,y)}{\partial y} \right).
\end{equation*}

Note that $\alpha(x,y)$ and $\beta(x,y)$ are functions of the direction field $\vec{d}_\theta(x,y)$ only (we can compute them from the network geometry).

We can do this coordinate transformation, since the flow evolves only along lines of constant $\eta$ in $(\xi,\eta)$-space, as it was already shown in \cite{TUM19c}. In $(\xi,\eta)$-plane the flow evolves along the $\xi$ coordinates, which are tangent to the flow motion, while in the orthogonal direction of $\eta$ there is no motion. 

Let us rescale the functions as $\bar\rho = \rho/(\alpha\beta)$, $\bar\rho_{max} = \rho_{max}/(\alpha\beta)$, $\bar v_{max} = v_{max} \alpha$, $\bar{\phi}_{max} = \phi_{max}/\beta$, $\bar{S}_{out} = S_{out}/\beta$ and $\bar{D}_{in} = D_{in}/\beta$.
Using \eqref{Rotation_mat}, \eqref{Scaling_matrix} with \eqref{beta} and \eqref{alpha} we compute the divergence of flow as \cite{TUM19c}. Thus, the model in $(\xi,\eta)$-space for the case of curved integral lines reads $\forall (\xi,\eta,t) \in D \times \mathbb{R}^{+}$:
 
\begin{equation}\label{Final_LWR}
\left\{
\begin{aligned}
& \frac{\partial \bar\rho (\xi, \eta, t)}{\partial t} + \frac{\partial (\bar \Phi (\xi,\eta, \bar\rho) }{\partial \xi} = 0, \\ 
& \bar{\phi}_{in}(\eta, t) = \min \left( \bar{D}_{in} (\eta, t), \bar{S} \left( \bar{\rho} \left( \xi_{min} (\eta), \eta, t \right) \right) \right), \\
& \bar{\phi}_{out}(\eta, t) = \min \left( \bar{D} \left( \bar{\rho} \left( \xi_{max} (\eta), \eta, t \right) \right), \bar{S}_{out} \left(\eta, t \right) \right), \\
& \bar\rho(\xi,\eta,0) = \bar\rho_0 (\xi,\eta),
\end{aligned} 
\right.
\end{equation}
where the flow function $\bar{\Phi}(\xi,\eta,\bar \rho)$ is now a scalar: 
\begin{equation}\label{Greenshields_xi_eta}
\bar{\Phi}(\xi,\eta,\bar \rho) = \bar v_{max}(\xi,\eta) \bar{\rho}(\xi,\eta,t) \left( 1 - \frac{\bar{\rho}(\xi,\eta,t)}{\bar \rho_{max}(\xi,\eta)}\right). 
\end{equation}

Equation \eqref{Final_LWR} has absolutely the same structure as \eqref{IBVP_0} with the difference only in rescaled state and parameters of FD. Thus, the result of Theorem 1 can be directly applied to control the system \eqref{Final_LWR}, taking into account the rescaling procedure.

\section{Numerical example}


\subsection{Urban network}

\begin{figure*}
\begin{center}
\includegraphics[width=16.0cm]{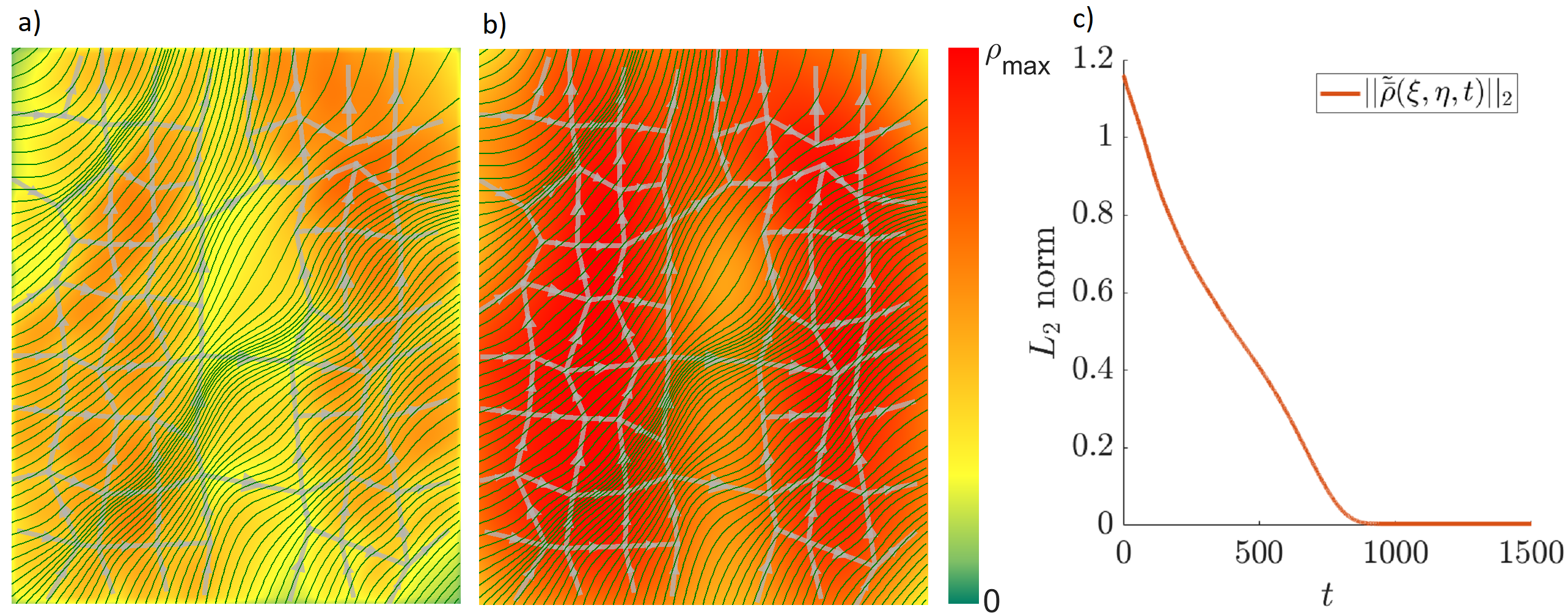}    
\caption{a) The desired state; b) initial congested network; c) the $L_2$ norm of the density error as a function of time.} 
\label{fig:Result}
\end{center}
\end{figure*}

As an example network we consider a $10 \times 10$ Manhattan grid of $1$ ${km}^2$. Positions of nodes (intersections) are slightly disordered with the white noise of standard deviation $20$ $m$ (denoted in grey on Figure \ref{fig:Result} in the middle and left). We assume that all roads are single-lane and are globally oriented towards the North-East direction. The network contains a topological bottleneck in the middle, e.g., a river with some bridges. The speed limits on most of the roads are set to $30$ $km/h$, and there are a few roads with $50$ $km/h$. 

For the computation of direction field $\vec{d}_{\theta}$ we chose such a sensitivity parameter $\beta = 20$ (see Section II, A) that the flux follows only the global trend of direction of all roads (North-East).

\subsection{Result}

Now we will consider the fully congested network with:
$$
\rho_0 (\xi, \eta) = \rho_{max} (\xi, \eta) \quad \forall (\xi, \eta) \in D.
$$

This state is also characterized by a maximal possible inflow, that the system can take, i.e. $\bar{\phi}_{in} (t) = \bar{\phi}_{max} \left( \xi_{min}, \eta, t \right)$. For the simulation, we first discretize the domain by $\eta$ and then implement the Godunov scheme for every constant $\eta$. For the outflow in the uncontrolled case we have created the so-called ghost cell. This means that for every $\eta$ we set the outflow equal to the flux on the previous cell of $\xi$. Thus, we obtain a congested scenario as illustrated in Fig. \ref{fig:Result}b.

The desired state is obtained by extracting from the minimal capacity along each $\eta$ minus $\epsilon = 10^{-5}$, which is depicted on Fig. \ref{fig:Result}a. Notice that the desired state flow is constant for every $\eta$ (steady-state), while the desired density is $\xi$-dependent. I

Finally, the fully congested state can be driven to the desired one by setting \eqref{Result} at the exit of each road of the network. Thereby, we observed that the spatial $L_2$-norm of the density error $\tilde{\bar{\rho}} \left( \xi, \eta, t \right)$ converges to zero as time goes to infinity, as it is shown in Fig. \ref{fig:Result}c.

\section{CONCLUSIONS}

In this paper we considered large-scale urban networks and designed the boundary control by actuating the supply at the domain's exit. The traffic state is given by the vehicle's density governed by the 2D-LWR model with space-dependent FD. Our designed controller is able to drive a heavily congested network with very large constant demand at the domain's entry to the $\epsilon$-optimal state w.r.t. the throughput maximization. We do this purely analytically relying only on the network topology. The topology is used to reconstruct the direction field, maximal velocities and road capacities on a continuous plane. The main trick was to perform the coordinate transformation allowing us to rewrite a 2D system as a set of 1D systems, and there to solve the control task.


\section*{ACKNOWLEDGEMENT}

The Scale-FreeBack project has received funding from the European Research Council (ERC) under the European Union’s Horizon 2020 research and innovation programme (grant agreement N 694209).

\bibliographystyle{IEEEtran}


\end{document}